\documentclass[10pt,leqno,twoside]{amsart}

\usepackage{tikz, subfigure, xcolor} 
\usetikzlibrary{snakes}
\usepackage{amsmath}
\usepackage{amsthm}
\usepackage{amssymb}
\usepackage{amsfonts}
\usepackage{cite}
\usepackage{dsfont}
\usepackage{hyperref}
\usepackage{doi}
\usepackage{comment}

\newtheorem{theorem}{Theorem}[section]
\newtheorem{lemma}[theorem]{Lemma}
\newtheorem{proposition}[theorem]{Proposition}

\theoremstyle{definition}
\newtheorem{definition}[theorem]{Definition}
\newtheorem{remark}[theorem]{Remark}
\newtheorem{remarks}[theorem]{Remarks}

\numberwithin{equation}{section}

\newcommand{\supp}{\mathrm{supp}}      
\renewcommand{\div}{\mathrm{div}\,}    

\providecommand{\norm}[1]{\lVert#1\rVert} 

\newcommand{\R}{\mathbb{R}}

\newcommand{\Z}{\mathbb{Z}}
\newcommand{\N}{\mathbb{N}}

\newcommand{\PP}{\mathbb{P}}

\newcommand{\IE}{\mathbb{E}}

\newcommand{\cA}{{\mathcal A}}

\newcommand{\cF}{{\mathcal F}}

\newcommand{\cH}{{\mathcal H}}

\newcommand{\cR}{{\mathcal R}}
\newcommand{\cS}{{\mathcal S}}
\newcommand{\cT}{{\mathcal T}}

\usepackage{eucal}

\newcommand\restr[2]{{
  \left.\kern-\nulldelimiterspace 
  #1 
  \vphantom{\big|} 
  \right|_{#2} 
  }}

\title{Partial and full hyper-viscosity for Navier-Stokes and primitive equations}


\author[Amru Hussein]{Amru Hussein} 
\address{Department of Mathematics,
	TU Darmstadt, Schlossgartenstr. 7, 64289 Darmstadt, Germany}
\email{hussein@mathematik.tu-darmstadt.de}

\begin{document}

\begin{abstract}
\noindent
The $3$-D primitive equations and incompressible Navier-Stokes equations with full hyper-viscosity and only horizontal hyper-viscosity 
are considered on the torus, i.e., the diffusion term $-\Delta$ is replaced by $-\Delta+ \varepsilon(-\Delta)^{l}$ or by $-\Delta + \varepsilon(-\Delta_H)^{l}$,  respectively, where $\Delta_H = \partial_x^2+\partial_y^2 $, $\Delta= \Delta_H + \partial_z^2$, $\varepsilon> 0$, $l>1$. Hyper-viscosity is applied in many numerical schemes, and in particular horizontal hyper-viscosity appears in meteorological models. A classical result by Lions states that for the Navier-Stokes equations uniqueness of global weak solutions for initial data in $L^2$ holds if $-\Delta$ is replaced by $(-\Delta)^{5/4}$. Here, for the primitive equations the corresponding result is proven for $(-\Delta)^{8/5}$. For the case of horizontal hyper-viscosity $l=2$ is sufficient in both cases.  Strong convergence for $\varepsilon\to 0$ of hyper-viscous solutions to a weak solution of the Navier-Stokes and primitive equations, respectively, is proven as well. The approach presented here is based on the construction of strong solutions via an evolution equation approach for initial data in $L^2$ and  weak-strong uniqueness.
\end{abstract}

\maketitle

\section{Introduction}
The subject of this article are the incompressible $3$-D  Navier-Stokes  and the $3$-D primitive equations -- for both of which uniqueness of weak solutions for initial data in $L^2$ is not known so far -- and the stabilizing effect of full and partial hyper-viscosity on these. 
The Navier-Stokes equations are a fundamental model for viscous fluids,
 and the primitive equations for the ocean and atmosphere are
 a model for geophysical flows derived from the Navier-Stokes equations 
assuming a hydrostatic balance for the pressure term in the vertical direction. For simplicity, a periodic setting is considered here.

In some numerical simulations hyper-viscous models are used instead of the classical ones, where  
to the viscosity term $-\Delta = -(\partial_x^2 + \partial_y^2+ \partial_z^2)$ of the classical models  higher powers $(-\Delta)^{l}$,  $l>1$, are added referred to as  hyper-viscosity or hyper-diffusion. In particular in some meteorological simulations  only horizontal hyper-viscosity terms are added replacing $-\Delta$  by partial hyper-viscous terms $-\Delta + (-\Delta_H)^l$,  $l>1$, where $\Delta_H=\partial_x^2+\partial_y^2$, cf. \cite{Lauritzenetall}.
The strategy of this regularization is to strengthen the  stabilizing effect of the linear part in order to balance turbulent effects of the non-linearity.
The idea of regularization by adding hyper-viscosity goes back to Ladyzhenskaya, see \cite{Lady1958}, where the hyper-viscous $3$-D Navier-Stokes equations with $-\Delta$ replaced  by $(-\Delta)^{2}$ is considered proving that this enforces uniqueness of weak solutions. This result has been refined by Lions, see \cite{Lions1959, Lions1969}, proving that $(-\Delta)^{5/4}$ is sufficient for this purpose. A recent result by Luo and Titi proves the sharpness of this result for a larger class of weak solutions, which do not necessarily satisfy the energy inequality, using convex integration, see \cite{LuoTiti2018}.

The overall aim of this note 
is to give a rigorous justification  of the usage of full and partial hyper-viscosity in numerical simulations by proving global well-posedness of these models for initial data in $L^2$. Global well-posedness comprises existence and uniqueness of global solutions as well as continuous dependence on the data. More specifically, first, a result is given for the primitive equations  corresponding to Lions' $(-\Delta)^{5/4}$-result for the Navier-Stokes equations, which seems to be lacking so far. For the primitive equations it turns out that $(-\Delta)^{8/5}$ is sufficient to enforce uniqueness of weak solutions, and it is not surprising that for the primitive equations a higher power appears since the non-linearity involves 'stronger' terms -- compared to the non-linearity of the Navier-Stokes equations -- such as
\begin{align*}
w(v)\partial_z v, \quad w(v)= - \int_{-1}^z \partial_x v_1 + \partial_y v_2.
\end{align*}
Second, the regularization by only horizontal hyper-viscosity for both the Navier-Stokes and the primitive equations is  studied. This feature is used in several meteorological models, see \cite[Chapter 13]{Lauritzenetall}, 
for instance in the COSMO model applied by several weather forecasting services, cf. \cite[Section 5, Numerical Smoothing]{COSMO}. Here, it is proven that the horizontally bi-harmonic hyper-viscosity $-\Delta + (-\Delta_H)^2$ is sufficient for both equations to assure uniqueness of weak solutions.  

Third, the relation of solutions to the hyper-viscous equations and non-hyper-viscous weak solutions is investigated.
Considering hyper-viscous terms for the above-mentioned values of $l>1$
\begin{align*}
A= -\nu \Delta +  \varepsilon(-\Delta)^{l}   \quad \hbox{and} \quad A = -\nu \Delta + \varepsilon(-\Delta_H)^{l} 
\end{align*}
with a hyper-viscosity parameter $\varepsilon>0$ and viscosity $\nu>0$, strong convergence for $\varepsilon\to 0+$ is proven. This implies that by the hyper-viscous regularization at least one weak solution of the classical model can be approximated.

The methods used here are based on an evolution equation approach for semi-linear equations. The strategy is to construct global strong solutions via methods from the theory of maximal $L^2$-regularity with  initial  values in $L^2$ and then to establish weak-strong uniqueness results. For the proof of local strong well-posedness, one considers for the linear part $A$ of the hyper-viscous equation fractional powers $A^{s}$, $s\in \R$, and interpolation-extrapolation scales. Then one can take the domain $D(A^{-1/2})$ as ground space, where the operator $A$ is a self-adjoint operator with domain $D(A^{1/2})$. Whether the portion of hyper-viscosity is sufficient for the equation 
\begin{align*}
\partial_t u + A u = \cF(u,u), \quad u(0)=u_0,
\end{align*}
is determined by the estimate on the quadratic non-linearity $\cF(u,u)$ with respect to the $D(A^{-1/2})$-norm, where a quadratic estimate in $D(A^{1/4})$ is sufficient to construct a strong solution with initial values in $L^2$, compare Figure~\ref{fig:DA}. The factor $1/4$ results form the quadratic estimate on the non-linearities
\begin{align*}
\norm{\cF(u,u)}_{L^2(0,T;D(A^{-1/2}))} \leq C \norm{v}^2_{L^4(0,T;D(A^{1/4}))}
\leq C \norm{v}^2_{H^{1/4}(0,T;D(A^{1/4}))}, \quad T>0,
\end{align*} 
 where one proceeds by using the mixed derivative theorem, cf. \cite[Corollary III.4.5.10]{PruessSimonett2016}, 
 \begin{align*}
 H^{1}(0,T;D(A^{-1/2}))\cap L^{2}(0,T;D(A^{1/2})) \hookrightarrow 
 H^{\theta}(0,T;D(A^{1/2-\theta})), \quad \theta\in (0,1),
 \end{align*}
for $\theta=1/4$, to control the non-linearity in the maximal $L^2$-regularity norm.
The weak-strong uniqueness follows then by using the additional regularity of  the strong solution, which allows one to use it as test-function, and the energy inequality for weak solutions. 
The convergence results are proven using a compactness argument, and therefore, precise convergence rates are not obtained.
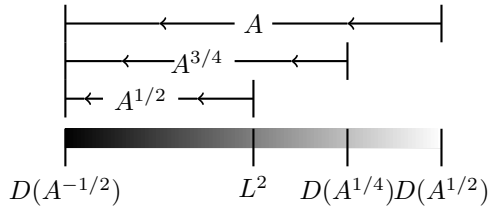
\begin{figure}[h]
	\begin{tikzpicture}[scale=0.5]
	\draw (0, 0) rectangle (10, 0.5); \
	\shade[left color=black, right color=white]
	(0,0) rectangle (10,0.5);
	\draw[thick] (0,0.5) -- (0,-0.5);
	\draw[thick] (0,2.8) -- (0,3.8);
	\node[below] at (0,-0.5) {$D(A^{-1/2})$}; 
	\draw[thick] (5,0.5) -- (5,-0.5);
	\node[below] at (5,-0.5) {$L^2$};
	\draw[thick] (7.5,0.5) -- (7.5,-0.5);
	\node[below] at (7.5,-0.5)
	{$D(A^{1/4})$};
	\draw[thick] (10,0.5) -- (10,-0.5);
	\node[below] at (10,-0.5)
	{$D(A^{1/2})$};
	\draw[thick] (10,3.8) -- (10,2.8);
 \draw[thick, ->] (5,1.3) -- (3.5,1.3);
  \draw[thick] (3.5,1.3) -- (3,1.3);
 \draw[thick, ->] (1,1.3) -- (0.5,1.3);
 \draw[thick] (0.5,1.3) -- (0,1.3); 
 	\draw[thick] (5,1.8) -- (5,0.8);
 	\draw[thick] (0,1.8) -- (0,0.8);
 	\node at (2,1.3){$A^{1/2}$}; 
    \draw[thick] (6,2.3) -- (4.5,2.3);
    \draw[thick, ->] (7.5,2.3) -- (6,2.3);
  \draw[thick, ->] (2.9,2.3) -- (1.5,2.3);
  \draw[thick] (1.5,2.3) -- (0,2.3);
	\node at (3.5,2.3){$A^{3/4}$};
	\draw[thick, -] (0,3.3) -- (2.5,3.3);
	\draw[thick, <-] (2.5,3.3) -- (4.5,3.3);
		\draw[thick] (7.5,2.8) -- (7.5,1.8);
		\draw[thick] (0,2.8) -- (0,1.8);
	\node at (5,3.3)
	{$A$};
	\draw[thick, -] (5.5,3.3) -- (7.5,3.3);
	\draw[thick, <-] (7.5,3.3) -- (10,3.3);
	\end{tikzpicture}
	\caption{Operator domains and interpolation-extrapolation scales}\label{fig:DA}
\end{figure}

For the $3$-D Navier-Stokes equations existence of weak solutions for initial data in $L^2$ is a classical result due to Leray, cf. \cite{Leray1934}, while
uniqueness holds locally and globally for small data for initial data in $L^3$, cf. the classical results by Giga  \cite{Gig86} and Kato \cite{Kat84}. There are several improvements including $u_0\in BMO^{-1}$ by Koch and Tataru, see \cite{Koch2001}, or $u_0$ in certain Besov spaces by Cannone, see \cite{Can97}. In fact for a larger and less regular class of weak solutions  non-uniqueness for initial data in $L^2$ has been proven recently by Buckmaster and Vicol using techniques from convex integration \cite{Buckmaster2017}. However, it is still an open question if this holds for weak solutions satisfying the energy inequality.

There are many analytical results concerning the hyper-viscous  Navier-Stokes equations while the hyper-viscous primitive equations have not been addressed so far to the best knowledge of the author. The asymptotic behavior for $t\to \infty$ of some hyper-viscous solutions compared to non-hyper-viscous ones has been investigated by Cannone and Karch, see \cite{Cannone2005}. The study of global attractors and strong convergence of hyper-viscous solutions has been performed by Ou and Sritharan, see
\cite{Ou1991, Ou1996}, and 
Younsi, see \cite{Younsi2010} and the references therein. 
Weak convergence of hyper-viscous solutions has been studied already by Lions \cite{Lions1969}. The question if hyper-viscosity has always a stabilizing effect is addressed in \cite[Section 5]{Larios2016} considering cases where energy inequalities or maximum principles are not preserved by the hyper-viscous model. Only horizontal hyper-diffusion $\Delta_H^2$ for the planetary geostrophic equations has been proposed in \cite{SamelsonVallis1997} and analyzed in \cite{CaoTitiZiane2004}.
Spectral hyper-viscosity, where hyper-viscosity is dependent on the frequency range, is studied for instance in \cite{GuermondPrudhomme2003} and \cite{AvriXiao2009}, see also the references therein. Regularity assumptions  -- on the lines of classical results for the non-hyper-viscous Navier-Stokes equations -- which assure uniqueness of weak solutions for any $
\alpha>1$ are investigated in \cite{DingSun2015}. 

Another class of modifications of the Navier-Stokes equations are the so-called $\alpha$-models, see \cite{Ali2013} and the references therein. 
An example is the simplified Bardina turbulence model, which involves some additional smoothing, and global existence and uniqueness of weak solutions and its global attractor are studied in \cite{Yanpingetall2006}. The Lagrangian averaged Navier-Stokes-$\alpha$-models are generalizations of the hyper-viscous Navier-Stokes equations, see \cite{OlsenTiti2007} and the references therein.

Note that in contrast to the Navier-Stokes equations, it is known that the $3$-D primitive equations admit a unique, global, strong solution for arbitrary large data $v_0\in H^1$. This breakthrough result
has been proven in 2007 by Cao and Titi \cite{CaoTiti2007}. 
Existence of weak solution for initial data in $L^2$ goes back to the pioneering work by Lions, Temam and Wang, see \cite{Lionsetall1992, Lionsetall1992_b}, while uniqueness for initial data in $L^2$ still constitutes an open problem. 
However, some progress has been made in this direction, and it has been proven that uniqueness of weak solutions holds for $v_0\in L^{\infty}_{xy}L^1_z$, see \cite{GigaGriesHusseinHieberKashiwabara2017NN}, or for $v_0\in L^6$ with $\partial_z v_0\in L^2$ and $v_0\in C^0$, see \cite{LiTiti2015} and also \cite{LiTiti2016} for a recent survey on analytical results for the primitive equations.

This note is organized as follows: In Section~\ref{sec:pre} the hyper-viscous equations are given for a periodic setting and basic notions such as function spaces and weak solutions are made precise. The main results are formulated in Section~\ref{sec:main} with proofs in Section~\ref{sec:proofs}.

\section{Preliminaries}\label{sec:pre}

Consider the cylindrical domain $\Omega :=(0,1)\times (0,1)\times (-1,1)$, and denote by $x,y\in (0,1)\times (0,1)$ the horizontal coordinates and by $z\in (-1,1)$ the vertical one. Let $u=(v,w)\colon \Omega\rightarrow \R^3$ be the velocity with vertical component $v=(v_1,v_2)$ and horizontal component $w$, and $p\colon \Omega \rightarrow \R$ the pressure. Then 
the hyper-viscous Navier-Stokes equations are given by
\begin{equation}\label{eq:NS}
\left \{\begin{array}{rll}
\partial _t u +u\cdot\nabla u +A  u+\nabla p=& \ 0&\text{ in }(0,T)\times\Omega ,\\
\div u =& \ 0&\text{ in }(0,T)\times\Omega ,\\
u,p\text{ periodic in }x,y,z, & &\\
u (0)=&\ u_0&\text{ in }\Omega, 
\end{array}\right .\tag{HNS$_{A}$}
\end{equation}
and 
the hyper-viscous primitive equations by
\begin{equation}\label{eq:PE}
\left \{\begin{array}{rll}
\partial _tv+u\cdot\nabla v + A v+\nabla _Hp=& \ 0&\text{ in }(0,T)\times\Omega ,\\
\partial _zp=&\ 0&\text{ in }(0,T)\times\Omega ,\\
\div u=& \ 0&\text{ in }(0,T)\times\Omega ,\\
p\text{ periodic in }x,y& 
\\
v ,w \text{ periodic in }x,y,z,&\text{ even and odd }&\text{ in }z,\\
u(0)=&\ u_0&\text{ in }\Omega ,
\end{array}\right .\tag{HPE$_{A}$}
\end{equation}
where $\nabla_H  = (\partial_x, \partial_y)^T$, are considered for the cases
\begin{align}\label{eq:opA}
A= -\nu \Delta - \varepsilon\Delta^{l}   \quad \hbox{and} \quad A = -\nu \Delta - \varepsilon\Delta_H^{l}, \quad \hbox{for some}\quad \nu,\varepsilon\geq 0, \quad l> 1,
\end{align}
where $\Delta = \partial_x^2 + \partial_y^2+ \partial_z^2$ and $\Delta_H =\partial_x^2 + \partial_y^2$.
With a slight abuse of notation the operators in \eqref{eq:PE} and \eqref{eq:NS} are denoted both by $A$. 
Note that the vertical periodicity and parity conditions in \eqref{eq:PE} 
correspond in to an equivalent set of equations on $(0,1)\times (0,1)\times (-1,0)$ with lateral periodicity and
\begin{align*}
\partial_z v\vert_{z=0} = \partial_z v\vert_{z=-1} =0 \quad \hbox{and} \quad w\vert_{z=0} = w\vert_{z=-1} =0.
\end{align*}
Therefore, the divergence free condition 
for the primitive equation translates into 
$\div _H\overline{v} =0$, where $\overline{v}(x,y)=\frac{1}{2}\int^1_{-1}v(x,y,z)\mathop{}\!\mathrm{d}z$, and 
\[w(\cdot,\cdot,z)= -\int^z_{-1}\div _H v(\cdot,\cdot,\zeta)\mathop{}\!\mathrm{d}\zeta, \quad \hbox{where}\quad  \div _H v = \partial_x v_1+ \partial_y v_2.
\]
Hence the primitive equations can be reformulated equivalently using the surface pressure $p_s\colon (0,1)\times(0,1)\rightarrow \R$ and $u=u(v)=(w(v),v)$ to become
\begin{equation}\label{eq:PE2}
\left \{\begin{array}{rll}
\partial _tv+w(v)\partial_z v + v \cdot \nabla_H v + A v+\nabla _Hp_s=& \ 0&\text{ in }(0,T)\times\Omega ,\\
\div_H \bar{v}=& \ 0&\text{ in }(0,T)\times\Omega ,\\
p_s\text{ periodic in }x,y& 
\\
v \text{ periodic in }x,y,z,&\text{ even }&\text{ in }z,\\
v(0)=&\ v_0&\text{ in }\Omega.
\end{array}\right .\tag{HPE$_{A}'$}
\end{equation}

\subsection{Function spaces}
For $s\in [0,\infty )$ one defines the Bessel potential spaces
\[H^{s}_{per}(\Omega )=\overline{C_{per}^\infty (\overline{\Omega })}^{\Vert \cdot\Vert _{H^{s}}} \quad \hbox{and} \quad 
H^{-s}_{per}(\Omega ) =  (H^{s}_{per}(\Omega ))', 
\]
where
$C_{per}^\infty (\overline{\Omega })$ denotes the space of smooth functions which are periodic of any order (cf. \cite[Section 2]{HieberKashiwabara2015}) in all three directions 
on $\partial \Omega$ and $(\cdot)'$ denotes the $L^2$-dual. 
The space $H^{s}(\Omega )$ denotes the Bessel potential space of order $s$, with norm $\Vert \cdot \Vert _{H^{s}}$ defined via the restriction  to $\Omega$ of the corresponding space defined on the whole space (cf. \cite[Definition 3.2.2.]{Triebel}). 
One sets $H^{0}=L^2$, and moreover $H^s_{per}(\Omega)$ for $s\in \R$ can be characterized by means of Fourier series as
\begin{align*}
H^s_{per}(\Omega)= \{v \hbox{ such that } \sum_{k\in \Z^3}(|k|^2+1)^{s} |\hat{v}(k)|^2 <\infty  \},
\end{align*}
where
\begin{align*}
\hat{v}(k) = \frac{1}{2} \int_{\Omega} e^{2\pi ik_1 x}e^{2\pi ik_2 y} e^{\pi ik_2 z} dx\, dy\, dz, \quad  k=(k_1,k_2,k_3)\in \Z^3. 
\end{align*}

The divergence free conditions in the above sets of equations can be encoded into the space of solenoidal functions
\[
L^2_\sigma (\Omega )=\overline{\{u\in C_{per}^\infty (\overline{\Omega })^3:\div u =0\}}^{\Vert \cdot\Vert _{L^2}}\mbox{ and } 
L^2_{\overline{\sigma }} (\Omega )=\overline{\{v\in C_{per}^\infty (\overline{\Omega })^2, \quad v\hbox{ even w.r.t. $z$}:\div _H\overline{v} =0\}}^{\Vert \cdot\Vert _{L^2}},
\]
and there are bounded linear projections -- the \textit{Helmholtz projection} and the \textit{hydrostatic Helmholtz projection} --
\begin{align*}
\PP_{\sigma}\colon L^2(\Omega )^3 \rightarrow L^2_\sigma (\Omega ) \quad \hbox{and} \quad \PP_{\bar{\sigma}}\colon L_{ev}^2(\Omega )^2:= \{v\in L^2(\Omega )^2\colon v\hbox{ even w.r.t. $z$}\} \rightarrow L^2_{\bar\sigma} (\Omega ),
\end{align*}
respectively. In the periodic setting both projections can be given explicitly using $3$-D and $2$-D Riesz-transforms, i.e., 
\begin{align*}
\PP_{\sigma}v = v - (R_iR_j)_{1\leq i,j\leq 3}v \quad \hbox{and} \quad \PP_{\bar\sigma}v = \tilde{v} - (R_iR_j)_{1\leq i,j\leq 2}\bar{v}, \quad \tilde{v} = v-\bar{v}, 
\end{align*}
where $R_j$ is defined by the symbol $r_j(k)=\frac{k_j}{|k|}$. In particular $\PP_{\sigma}$ and $\PP_{\bar\sigma}$ extend to bounded projections in $H^s_{per}(\Omega)$ for any $s\in \R$. Therefore, one sets for  $s\in \R$
\begin{eqnarray*}
H^{s}_{per, \sigma}(\Omega)= \PP_{\sigma}H^{s}_{per}(\Omega)^3  & \hbox{and}&
 H^{s}_{per, \bar\sigma}(\Omega)= \PP_{\bar{\sigma}}\{v\in H^{s}_{per}(\Omega)^2\colon v\hbox{ even w.r.t. $z$} \},  
\end{eqnarray*}
where for $s\geq 0$ one has $H^{s}_{per, \sigma}(\Omega)= H^{s}_{per}(\Omega)^3  \cap L^2_{\sigma}(\Omega)$ and $H^{s}_{per, \bar\sigma}(\Omega)= H^{s}_{per}(\Omega)^2  \cap L^2_{\bar\sigma}(\Omega)$.

\subsection{Operator domains, interpolation and extrapolation scales}
Note that due to periodicity the projectors $\PP_{\sigma}$ and $\PP_{\bar{\sigma}}$ commute with the operators $A$ given by \eqref{eq:opA} while annihilating the pressure terms. Hence \eqref{eq:NS} and \eqref{eq:PE} are of the semi-linear form
\begin{align}\label{eq:Fform1}
\partial_t \psi + A \psi + \cF(\psi,\psi) =f, \quad \psi(0)=\psi_0,
\end{align}
where for \eqref{eq:NS} $\psi=u$ and for \eqref{eq:PE} $\psi=v$ with
\begin{align}\label{eq:Fform2}
\cF(u,u')=\cF_{NS}(u,u')= \PP_{\sigma}(u\cdot \nabla u') \quad \hbox{and} \quad  \cF(v,v')=\cF_{PE}(v,v')= \PP_{\bar\sigma}(w(v)\partial_z v' + v\cdot \nabla_H v'),
\end{align}
respectively, and with an abuse of notation one omits the projections in the linear part writing $A=\PP_{\sigma}A=A\PP_{\sigma}$ and $A=\PP_{\bar\sigma}A=A\PP_{\bar\sigma}$.
The mapping properties of the operators $A$ in $L^2$-spaces can be studied explicitly using Fourier series. Therefore, for
\begin{align}\label{eq:Afull}
A= -\nu \Delta + \varepsilon(-\Delta)^{l},  \quad \nu\geq 0,\quad \varepsilon>0, \quad l\geq 1,
\end{align}
one has in the case of Navier-Stokes and primitive equations, respectively, that $A\colon D(A)\subset L_{\sigma}^2(\Omega) \rightarrow L_{\sigma}^2(\Omega)$ and $A\colon D(A)\subset L_{\bar\sigma}^2(\Omega) \rightarrow L_{\bar\sigma}^2(\Omega)$ with 
\begin{align*}
D(A)= H^{2l}_{per, \sigma}(\Omega)   \quad \hbox{and} \quad D(A)= H^{2l}_{per, \bar\sigma}(\Omega), 
\end{align*}
define self-adjoint operators. This allows one to define fractional powers, where 
\begin{align*}
D(A^s)= H^{2s}_{per, \sigma}(\Omega) \quad \hbox{and} \quad D(A^s)= H^{2s}_{per, \bar\sigma}(\Omega), \quad s\in \R,
\end{align*}
respectively.
In the case 
\begin{align}\label{eq:Ahor}
A= -\nu \Delta + \varepsilon(-\Delta_H)^{l},  \quad \nu> 0,\quad \varepsilon>0, \quad l\geq 1,
\end{align}
these define self-adjoint operators in $L_{\sigma}^2(\Omega)$ and $L_{\bar\sigma}^2(\Omega)$  with
\begin{align*}
D(A)&= \{u\in L^{2}_{\sigma}(\Omega)\colon \sum_{k \in \Z^3} (|k_3|^2 + |k_H|^{2l}+1)^2 |\hat{u}(k)|^2 <\infty \}, \\
D(A)&= \{v\in L^{2}_{\bar\sigma}(\Omega)\colon \sum_{k \in \Z^3} (|k_3|^2 + |k_H|^{2l}+1)^2 |\hat{v}(k)|^2 <\infty \},    
\end{align*}
respectively, where $k_H = (k_1,k_2)$. For $s\in \R$
\begin{align*}
D(A^s)&= \{u\hbox{ such that } \sum_{k \in \Z^3} (|k_3|^2 + |k_H|^{2l}+1)^s |\hat{u}(k)|^2 <\infty \hbox{ and }\PP_{\sigma}u=u\}, \\
D(A^s)&= \{v \hbox{ such that }  \sum_{k \in \Z^3} (|k_3|^2 + |k_H|^{2l}+1)^{s} |\hat{v}(k)|^2 <\infty \hbox{ and }\PP_{\bar\sigma}v=v \}. 
\end{align*}
Thereby one can define interpolation and extrapolation scales using fractional powers of $A$
\begin{align*}
A\colon D(A^{s+1})\subset D(A^{s}) \rightarrow D(A^{s}), \quad s\in \R,
\end{align*}
and in particular for $s=-1/2$, $A\colon D(A^{1/2})\subset D(A^{-1/2}) \rightarrow D(A^{-1/2})$,
where with an abuse of notation the operators are still denoted by $A$ although being defined between different spaces.

\subsection{Weak solutions}
Note that there are several notions of weak solutions depending on the regularity class. For instance aside the classical notion of Leray-Hopf  weak solutions used here, there is also the larger class of weak solutions used in the context of convex integration, cf. \cite{Buckmaster2017}.

\begin{definition}[Weak solution to the hyper-viscous Navier-Stokes equations]
	Let $A$ be as in \eqref{eq:Afull} or \eqref{eq:Ahor}, $T>0$, $u_0\in L^2_{\sigma}(\Omega)$ and $f\in L^2(0,T;D(A^{-1/2}))$. Then
a function $u$ is called \textit{weak solution to the  hyper-viscous Navier-Stokes equations} \eqref{eq:NS} if 
\begin{enumerate}
	\item[(i)] $u\in C^{w}([0,T]; L^2_{\sigma}(\Omega))\cap L^2(0,T; D(A^{1/2}))$;  
	\item[(ii)] For all $\varphi \in C^{1}([0,T]; D(A^{1/2}))\cap L^2(0,T; D(A))$ and $t\in (0,T)$
	\begin{align*}
	\int_0^t u \cdot \partial_t \varphi - A^{1/2}v\cdot A^{1/2}\varphi - (u\cdot \nabla u)\cdot \varphi = v(t)\cdot \varphi(t)-v(0)\cdot \varphi(0) - \int_0^t f\cdot \varphi;
	\end{align*}
	\item[(iii)] $u$ satisfies for all $t\in (0,T]$ and almost all $s\in (0,t)$  the strong energy inequality
	\begin{align}\label{eq:ei_ns}
	\norm{u(t)}^2_{L^2} + 2\int_s^t \norm{A^{1/2}u(\tau)}^2_{L^2} d\tau \leq 	\norm{u(s)}^2_{L^2} + \int_s^t f(\tau)\cdot u(\tau) d\tau. \tag{EI$_{NS}$}
	\end{align}
\end{enumerate}
\end{definition}

 \begin{definition}[Weak solution to the hyper-viscous primitive equations]
 	Let $A$ be as in \eqref{eq:Afull} or \eqref{eq:Ahor}, $T>0$, $v_0\in L^2_{\bar\sigma}(\Omega)$ and $f\in L^2(0,T;D(A^{-1/2}))$. Then
 	a function $v$ is called \textit{weak solution to the  hyper-viscous primitive equations} \eqref{eq:PE} if 
 	\begin{enumerate}
 		\item[(i)] $v\in C^{w}([0,T]; L^2_{\sigma}(\Omega))\cap L^2(0,T; D(A^{1/2}))$;
 		\item[(ii)] For all $\varphi \in C^{1}([0,T]; D(A^{1/2}))\cap L^2(0,T; D(A))$ and $t\in (0,T)$
 		\begin{align*}
 		\int_0^t v \cdot \partial_t \varphi - A^{1/2}v\cdot A^{1/2}\varphi - (u(v)\cdot \nabla v)\cdot \varphi = v(t)\cdot \varphi(t)-v(0)\cdot \varphi(0) - \int_0^t f\cdot \varphi;
 		\end{align*}
 		\item[(iii)] $v$ satisfies for all $t\in (0,T]$ and almost all $s\in (0,t)$ the strong energy inequality
 		\begin{align}\label{eq:ei_pe}
 		\norm{v(t)}^2_{L^2} + 2\int_s^t \norm{A^{1/2}v(\tau)}^2_{L^2} d\tau \leq 	\norm{v(s)}^2_{L^2} + \int_s^t f(\tau)\cdot v(\tau) \hbox{d} \tau.  \tag{EI$_{PE}$}
 		\end{align}
 	\end{enumerate}
 \end{definition}

Here, $C^{w}$ stands for spaces of weakly continuous functions. A weak solution is said to satisfy the \textit{energy equality} if in \eqref{eq:ei_ns} and
\eqref{eq:ei_pe} $\leq$ can be replaced by $=$, respectively. Note that for $T\in (0,\infty)$ it follows from the energy inequality that there is a constant $C>0$ such that
\begin{align} \label{eq:energyestimate}
\norm{u}^2_{L^{\infty}(0,T;L^2)} + \norm{u}^2_{L^{2}(0,T;D(A^{1/2}))} \leq C 
(\norm{u(0)}^2_{L^2} + \norm{f}^2_{L^{2}(0,T;D(A^{-1/2}))}),
\end{align}
where one uses the embedding $L^{\infty}(0,T;L^2(\Omega))\hookrightarrow L^{2}(0,T;L^2(\Omega))$ and Young's inequality amongst others. For the primitive equations the analogous statement holds. 

\section{Main results}\label{sec:main}

The main results are formulated in parallel for the hyper-viscous Navier-Stokes and primitive equations.  Note that case $(i)$ in Theorem~\ref{thm:hns} $(a)$ is Lions' classical result, see \cite{Lions1969}. The corresponding convergence in $(b)$ for $\delta=1$ has been proven by Guermond and Prudhomme, see \cite{GuermondPrudhomme2003}, and Younsi, see \cite{Younsi2010}. It is repeated here to provide a complete picture. Also the proofs given here are slightly different than the original ones since here they rely  mainly on evolution equation methods rather than energy methods. 

\begin{theorem}[Hyper-viscous Navier-Stokes equations]\label{thm:hns} Consider the hyper-viscous Navier-Stokes equations \eqref{eq:NS} with
\begin{align*}
		&(i)\quad A= \nu (-\Delta) + \varepsilon(-\Delta)^{5/4}, \quad \varepsilon>0, \quad \nu \geq 0 \quad \hbox{or} \\
		 &(ii)\quad A= \nu (-\Delta) + \varepsilon(-\Delta_H)^{2}, \quad \varepsilon>0, \quad \nu >0.
		\end{align*}
Let $u_0\in L^2_{\sigma} (\Omega )$ and $f\in L^2(0,T;D(A^{-1/2}))$, where  $T\in (0,\infty)$.
	\begin{itemize}
		\item[(a)] Then in both cases $(i)$ and $(ii)$ there is a unique weak solution $u$ to \eqref{eq:NS}. Moreover, $u$ satisfies the strong energy equality and 
		\begin{align*}
		u\in L^2(0,T; D(A^{1/2})) \cap H^1(0,T; D(A^{-1/2})).
		\end{align*} 
		\item[(b)] Let $A=A_n$ be as in $(i)$ provided $\nu>0$ or as in $(ii)$, where $\varepsilon_n>0$, $n\in \N$. For given $u_0\in L^2_{\sigma} (\Omega )$ and $f\in L^2(0,T;H^{-1}_{per,\sigma}(\Omega))$ let
		$u_{\varepsilon_n}$ be the solution to \eqref{eq:NS} from $(a)$. If $\varepsilon_n\to 0$  as $n\to \infty$, then there exists a subsequence $(\varepsilon_{n_k})$ and a weak solution $u$ with the same data $u_0$ and $f$ to the Navier-Stokes equations, i.e., \eqref{eq:NS} with $A=-\Delta$, such that
		\begin{align*}
		\norm{u_{\varepsilon_{n_k}}-u}_{L^2(0,T;H^{1-\delta}(\Omega))} \to 0 \hbox{ as } k \to \infty, \quad \delta \in (0,1].
		\end{align*}  
	\end{itemize}
\end{theorem}

Correspondingly for the  hyper-viscous primitive equations one has
\begin{theorem}[Hyper-viscous primitive equations]\label{thm:hpe}
	Consider the hyper-viscous primitive equations \eqref{eq:PE} 
with
\begin{align*}
		&(i)\quad A= \nu (-\Delta) + \varepsilon(-\Delta)^{8/5}, \quad \varepsilon>0, \quad \nu \geq 0 \quad \hbox{or} \\
		 &(ii)\quad A= \nu (-\Delta) + \varepsilon(-\Delta_H)^{2}, \quad \varepsilon>0, \quad \nu >0.
		\end{align*}	
	Let $v_0\in L^2_{\overline{\sigma}} (\Omega )$ and $f\in L^2(0,T;D(A^{-1/2}))$, where  $T\in (0,\infty)$.
	\begin{itemize}
		\item[(a)]  Then in both cases $(i)$ and $(ii)$ there is a unique weak solution $v$ to \eqref{eq:PE}. Moreover, $v$ satisfies the strong energy equality and 
		\begin{align*}
		v\in L^2(0,T; D(A^{1/2})) \cap H^1(0,T; D(A^{-1/2})).
		\end{align*}  
		\item[(b)]  Let $A=A_n$ be as in $(i)$ provided $\nu>0$ or as in $(ii)$, where $\varepsilon_n>0$, $n\in \N$. For given $v_0\in L^2_{\overline{\sigma}} (\Omega )$ and $f\in L^2(0,T;H^{-1}_{per,\bar\sigma}(\Omega))$
		let $v_{\varepsilon_n}$ be the solution to \eqref{eq:PE} from $(a)$. If $\varepsilon_n\to 0$  as $n\to \infty$, then there exists a subsequence $(\varepsilon_{n_k})$ and a weak solution $v$ with the same data $v_0$ and $f$ to the primitive equations, i.e., \eqref{eq:PE} with $A=-\Delta$, such that
		\begin{align*}
		\norm{v_{\varepsilon_{n_k}}-v}_{L^2(0,T;H^{1-\delta}(\Omega))} \to 0 \hbox{ as } k \to \infty, \quad \delta \in (0,1].
		\end{align*}  
	\end{itemize}
\end{theorem}

\begin{remarks}
	\begin{itemize}
	    		\item[(a)] Corresponding statements hold for all higher powers in the hyper-viscosity  terms. For the primitive equations it seems that hyper-viscosity  is only needed for the part with vanishing vertical average, i.e., rewriting \eqref{eq:PE} as a system for $\bar v$ and $\tilde{v}$, cf. \cite{CaoTiti2007}, one has to add hyper-viscous terms only in the equation for $\tilde{v}$.
\item[(b)] In part $(a)$ with $A$ as in $(ii)$ one can consider equivalently $\nu \partial_z^2 + \varepsilon \Delta_H^2$ with $\nu, \varepsilon>0$.
	    \item[(c)] The case $T=\infty$ can be included if for instance the data are mean value free.
	    \item[(d)] Note that $L^2(0,T; D(A^{1/2})) \cap H^1(0,T; D(A^{1/2})) \hookrightarrow C^0([0,T]; L^2(\Omega))$, cf. \cite[III.3.5 eq. (3.81)]{PruessSimonett2016}, and therefore the $L^2$-initial condition  is well-defined. 
	    \item[(e)] Since the proof is based on the construction of strong solutions via maximal $L^2$-regularity, one can prove using the so-called 'parameter-trick', cf. \cite{GigaGriesHieberHusseinKashiwabara2017} and the references given therein, that solutions become real analytic in time and space for $t>0$ provided the force terms has this property. For the periodic setting the solution is real analytic even on the boundary.
	\end{itemize}
\end{remarks}
Optimality of these results is of course related to the question of uniqueness of weak solutions which is not known so far. Therefore it is discussed only with respect to the method applied here. 
Considering the primitive equations, one can estimate the non-linearity as follows 
\begin{align*}
\norm{\div u(v)\otimes v}_{H^{-1}}\leq \norm{u(v)\otimes v}_{L^{2}} \leq \norm{v}_{H^{1/2}}\norm{v}_{H^{3/2}}.
\end{align*}
With a hyper-viscosity $(-\Delta^{3/2})$ -- which is less than $(-\Delta)^{8/5}$ given in Theorem~\ref{thm:hpe} -- this yields for $v_0\in L^2_{\bar\sigma}(\Omega)$ local strong solutions $v$ to \eqref{eq:PE} in the time-weighted space
\begin{align*}
v\in L^{2}_{\mu}(0,T; H_{per,\bar{\sigma}}^{3}(\Omega)) \cap v\in H^{1}_{\mu}(0,T; H^{-1}_{per,\bar{\sigma}}(\Omega)) \quad \hbox{with} \quad\mu=9/10,
\end{align*}
where $H^{-1}_{per,\bar{\sigma}}(\Omega)$ is taken as ground space.
For the notion of time-weighted spaces compare \cite[II.3.2.4]{PruessSimonett2016}. 
Time-weighted $L^2$-spaces for $\mu\in (0,1)$ are larger than the non-weighted $L^2$-space, i.e., where $\mu=1$. In particular, one cannot use this $v$ directly as test-function to prove the energy-equality or weak-strong uniqueness.

Also, one can ask if only vertical hyper-viscosity of the form
\begin{align*}
-\Delta_H + (-\partial_z^2)^{l}, \quad l>1,
\end{align*}
 rather than only horizontal hyper-viscosity, can be sufficient. Considering the Navier-Stokes equations one estimates the non-linearity  by
\begin{align*}
\norm{\div u\otimes u}_{D(A^{-1/2})}\leq C\norm{\div u\otimes u}_{H^{-1}}\leq \norm{u\otimes u}_{L^{2}} \leq \norm{u}^2_{H_{z}^{1/4}H_{xy}^{1/2}}.
\end{align*} 
Since $D(A^{1/4}) \subset L^2_zH^{1/2}_{xy}$, there is no freedom left to estimate the  last term by applying the mixed derivative theorem, and therefore the strategy applied for the case of horizontal hyper-viscosity is not applicable here. For the case of the primitive equations one encounters a similar situation.

\section{Proofs}\label{sec:proofs}

\subsection{Some estimates}
Taking advantage of the product structure of $\Omega = (0,1)^2 \times (-1,1)$, one introduces for $r,s\geq 0$ and $1\leq p,q \leq \infty$ the spaces
\begin{eqnarray*}
	H_z^{r,q}H^{s,p}_{xy}:= H^{r,q}((-1,1);H^{s,p}((0,1)^2))
\end{eqnarray*}
equipped with the norm $\norm{v}_{H_z^{r,q}H^{s,p}_{xy}}:= \norm{\norm{v(\cdot,z)}_{H^{s,p}((0,1)^2)}}_{H^{r,q}(-1,1)}$, where $H^{s,p}((0,1)^2)$ denotes the  Bessel potential space on $(0,1)^2$ defined as restriction of the Bessel potential space on $\R^2$, and
$H^{r,q}(-1,1;X)$ for a Banach space $X$ denotes the vector valued Bessel potential space on $(-1,1)$. Recall that for $r=0$ one sets $H^{r,q}=L^q$ and for $r\in \N$ one has $H^{r,q}=W^{r,q}$. Note that 
$$
H^{r+s,p}(\Omega) \subset  H_z^{r,p}H^{s,p}_{x,y}.
$$ 
Applying H\"{o}lder's inequality separately for vertical and horizontal components delivers
\begin{eqnarray*}
	\norm{fg}_{L^{q}_z L_{xy}^{p}} \leq \norm{f}_{L_z^{q_1}L_{xy}^{p_1}} \norm{g}_{L_z^{q_2}L_{xy}^{p_2}}, \quad 
	\tfrac{1}{p}=\tfrac{1}{p_1}+\tfrac{1}{p_2}, \quad \tfrac{1}{q}=\tfrac{1}{q_1}+\tfrac{1}{q_2}.
\end{eqnarray*}
Moreover, one obtains the embedding properties   
\begin{eqnarray*}
	H_z^{r,q}H^{s,p}_{xy} \hookrightarrow H_z^{r^{\prime},q^{\prime}}H^{s,p}_{xy}, &\hbox{whenever} & H_z^{r,q} \hookrightarrow H_z^{r^{\prime},q^{\prime}}, \\
	H_z^{r,q}H^{s,p}_{xy} \hookrightarrow H_z^{r,q}H^{s^{\prime},p^{\prime}}_{xy}, &\hbox{whenever} & H^{s,p}_{xy} \hookrightarrow H^{s^{\prime},p^{\prime}}_{xy}.
\end{eqnarray*}
Furthermore, one has the classical Sobolev embeddings, and by duality these carry over to negative scales such that for $s\leq 0$
\begin{align*}
L^p(\Omega)\hookrightarrow  H^{s}_{per}(\Omega) \quad \hbox{if} \quad s-\tfrac{3}{2} \leq-\tfrac{3}{p}, \quad p\in [1,\infty).
\end{align*}

A crucial ingredient appearing in  most of the proofs is the following
\begin{theorem}[Mixed derivative theorem, cf. Corollary III.4.5.10 in \cite{PruessSimonett2016}]
Let $X$ be a Banach space. Suppose that $A\in \cH^{\infty}(X)$ and $B\in \cR\cS(X)$ are commuting operators such that $\phi_A^{\infty}+\phi_{B}^R<\pi$. Then $A^{\alpha}B^{1-\alpha}(A+B)$ is bounded for each for each $\alpha\in (0,1)$, in particular
\begin{align*}
D(A) \cap D(B)= D(A+B) \hookrightarrow D(A^{\alpha}B^{1-\alpha}), \quad \alpha\in (0,1).
\end{align*}
\end{theorem}
For the notions $\cH^{\infty}(X),\cR\cS(X)$ and $\phi_A^{\infty}, \phi_{B}^R$ see \cite{PruessSimonett2016}. The mixed derivative theorem can be applied to time-space situations to prove the embedding
\begin{align}\label{eq:mdt_ts}
L^2(0,T;D(A))\cap H^{1}(0,T;X_0) \hookrightarrow H^{\theta}(0,T; D(A^{1-\theta})), \quad \theta\in (0,1),
\end{align}
where one sets $B=\partial_t$ on $L^2(\R;X_0)$ and uses an extension operator from $E_T\colon H^1(0,T;X_0)\rightarrow H^1(\R;X_0)$. Moreover, the assumptions of the mixed derivative theorem are satisfied for non-negative self-adjoint operators such as $(-\partial_z^2)^r$ and $(-\Delta_H)^s$,  defined on $L^2(\Omega)$ such that
\begin{align}\label{eq:mdt:xyz}
 H_z^rL_{xy}^2 \cap L_z^2H_{xy}^s \hookrightarrow H_z^{\theta r}H_{xy}^{(1-\theta) s}, \quad \theta \in (0,1), \quad r,s\geq 0.
\end{align}
Throughout the proofs generic constants $c,C>0$ are used.
 
\subsection{Semi-linear evolution equations and maximal $L^2$-regularity}
Consider a linear evolution equation 
\begin{align}\label{eq:lin}
\partial_t \psi + \cA \psi =  f, \quad \psi(0)=\psi_0,
\end{align}
in the separable Hilbert space $X_0$, where $\cA\colon X_1\subset X_0 \rightarrow X_0$ be closed and densely defined. For $T\in (0,\infty)$ one defines for $p\in (1,\infty)$ the $L^p$-ground space and the maximal $L^p$-regularity spaces by
\[
\IE _0(T):=L^p(0,T;X_0) \hbox{ and } \IE _1(T):=L^p(0,T;X_1)\cap H^{1,p}(0,T;X_0) \hookrightarrow C^0([0,T],X_{1-1/p,p}),
\]
respectively, with $X_{1-1/p,p}=(X_1,X_0)_{1-1/p,p}$, where $(\cdot,\cdot)_{\theta,p}$, $\theta\in (0,1)$, $p\in (1,\infty)$,  denotes the real interpolation functor.

One says that $\cA$ has \textit{maximal $L^p$-regularity} if for each $\psi_0\in X_{1-1/p,p}$
and $f\in \IE_0(T)$ there is a unique solution $\psi\in \IE_1(T)$ to \eqref{eq:lin} and $\norm{u}_\IE(T)\leq C (\norm{\psi_0}_{X_{1-1/p,p}}+ \norm{f}_{\IE_0(T)})$, $C>0$. Note that $C=C(T)$ can be chosen uniformly for all $0<T'\leq T$. In the following it is assumed that $p=2$ if not mentioned otherwise.

\begin{remark}[Maximal $L^2$-regularity in Hilbert spaces]\label{rem:Hilbert}
	The theory of maximal $L^p$-regularity extends to Banach spaces, cf. \cite{PruessSimonett2016} and the references therein. However, in the Hilbert space case there are some simplifications.	
	\begin{itemize}
		\item[(a)]  $\cA$ has maximal $L^2$-regularity for $T\in (0,\infty)$ if and only if it is the generator of a bounded analytic $C_0$-semigroup, cf. \cite[Theorem II.3.5.7]{PruessSimonett2016}.  
		\item[(b)] For the real interpolation functor $(\cdot,\cdot)_{\theta,p}$, one has 
		$$(X_1,X_0)_{\theta,2}=[X_1,X_0]_{\theta}=:X_{\theta}, \quad \theta\in (0,1),$$ where   $[\cdot,\cdot]_{\theta}$ denotes the complex interpolation functor, and if $\cA$ has a bounded $H^{\infty}$-calculus then $D(\cA^{\theta})=X_{\theta}$, $\theta\in (0,1)$, cf. \cite[Theorem 4.3.11]{Lunardi}. Hence $X_{1-1/2,2}=X_{1/2}$.
	\end{itemize}
	These statements apply in particular to non-negative, self-adjoint operators $\cA$.
\end{remark}

Consider the semi-linear problem for a given  bi-linearity $\cF(\cdot,\cdot)$
\begin{align}\label{eq:semilin}
\partial_t \psi + \cA \psi =  \cF(\psi,\psi) + f, \quad \psi(0)=\psi_0.
\end{align}
Here, a \emph{strong solution to \eqref{eq:semilin} on $(0,T^*)$} for $T^*\in (0,T]$, is a function $\psi\in \IE _1(T^*)$ such that  \eqref{eq:semilin} holds almost everywhere with $\psi(0)=\psi_0$.
Local strong well-posedness is proven by applying the contraction mapping principle. To this end, 
define the reference function $\psi_0^*\in \IE_1(T)$ for $\psi_0\in X_{1/2}$ as the solution of the inhomogeneous linear problem
	\begin{align*}
	\partial_t \psi +\cA \psi = f, \qquad \psi(0)=\psi_0, \quad \hbox{i.e.,} \quad  \psi_0^*(t)= e^{-t\cA}\psi_0 + \int_0^t e^{-(t-s)\cA}f(s) ds,
	\end{align*}
	and for $T^*\in (0,T]$,  $r>0$ the ball 
	\begin{align*}
	\mathds{B}_{r,T^*,\psi_0} := \{\psi\in \IE_1(T) \colon \psi(0)=\psi_0 \hbox{ and } \norm{\psi - \psi_0^*}_{\IE(T^*)} \leq r   \}  \subset\IE_1(T).
	\end{align*}
\begin{proposition}[Local well-posedness, compare \cite{PruessWilke2016}]\label{prop:loc_ex}
Let $\cA$ have maximal $L^2$-regularity, and $\cF(\cdot,\cdot)$ be bi-linear such that for some $C>0$ 
\begin{align*}
\norm{\cF(\psi,\psi')}_{X_0}\leq C \norm{\psi}_{X_{\beta_1}}\norm{\psi'}_{X_{\beta_2}}, \qquad \psi,\psi'\in X_1,
\end{align*}
where $\beta_1,\beta_2\in (1/2,1)$ with $\beta_1+\beta_2=3/2$.
Then for each 
\begin{align*}
\psi_0\in X_{1/2} \quad \hbox{and} \quad f\in \IE_0(T),
\end{align*}
there is  unique local solution $u\in \IE_1(T^*)$ to \eqref{eq:semilin} for some $T^*\in (0,T]$.
Moreover, there is a constant $C>0$ such that $\psi\in \mathds{B}_{r,T^*,\psi_0}$ if $r\leq 1/4C$ and $T^*$ is such that \begin{align*}
\norm{\psi_0^*}^2_{\IE_1(T)}< B(r,C):=\min\{r/4C, (r/4C)^2, (1/4C)^2\}.
\end{align*}
\end{proposition}
\begin{proof}
	The proof consists of a slight modification of the proof in \cite{PruessWilke2016}, where $\norm{\cF(\psi,\psi)-\cF(\psi',\psi')}_{X_0}\leq C \norm{\psi-\psi'}_{X_{\beta_1}}(\norm{\psi}_{X_{\beta_2}}+\norm{\psi'}_{X_{\beta_2}})$ is assumed. Here, one estimates for $p_k$, $k=1,2$, with
	$\frac{1}{p_k} = \beta_k-\frac{1}{2}$, where by assumption $\frac{1}{p_1}+\frac{1}{p_2}=\frac{1}{2}$, using first H\"older's inequality
	\begin{align*}
	\norm{\cF(\psi,\psi')}_{\IE_0(T)} &\leq C \norm{\psi}_{L^{p_1}(0,T;X_{\beta_1})} \norm{\psi'}_{L^{p_2}(0,T;X_{\beta_2})}  \\
	&\leq C \norm{\psi}_{H^{s_1}(0,T;X_{\beta_1})} \norm{\psi'}_{H^{s_2}(0,T;X_{\beta_2})}  \\
	&\leq C \norm{\psi}_{\IE_1(T)}\norm{\psi'}_{\IE_1(T)}, 
	\end{align*} 
		using second the embedding $H^{s_k}(0,T;X_{\beta_k}) \hookrightarrow L^{p_k}(0,T;X_{\beta_k})$ for $s_k=\frac{1}{2}-\frac{1}{p_k}$, $k=1,2$, and finally, the mixed derivative theorem $\IE_1(T)\hookrightarrow H^{\theta}(0,T; X_{1-\theta})$, $\theta\in (0,1)$, for $\theta = s_k=1-\beta_k$, $k=1,2$, cf. equation \eqref{eq:mdt_ts}. By density it follows that this estimate holds for all $\psi,\psi'\in \IE_1(T)$.
		
		Note that for a quadratic non-linearity, it is sufficient to estimate $\cF(\psi,\psi')$ since 
		\begin{align*}
		\cF(\psi,\psi)-\cF(\psi',\psi') = \tfrac{1}{2}\left(\cF(\psi-\psi',\psi) + \cF(\psi, \psi-\psi') + \cF(\psi-\psi',\psi') + \cF(\psi', \psi-\psi')  \right).
		\end{align*}
		
	Now one can proceed analogously to \cite{PruessWilke2016}. 
	Consider the map
	\begin{align*}
	\cT_{\psi_0} \colon  \mathds{B}_{r,T^*,\psi_0} \rightarrow \IE_1(T),  \quad \cT_{\psi_0} h = \psi,
	\end{align*}
	where $\psi$ is the unique solution in $\IE_1(T)$ to the linear problem 
	\begin{align*}
	\partial_t \psi + \cA\psi  = f + \cF(h,h), \qquad \psi(0)=\psi_0.
	\end{align*}	
	To prove that for suitable $r>0$ and $T^*\in (0,T]$ this defines a self-mapping, one estimates 
	\begin{align*}
	\norm{\psi-\psi_0^*}_{\IE_1(T)} &\leq C \norm{\cF(h,h)}_{\IE_0(T)} \leq  C \norm{(h-\psi_0^*)+\psi_0^*}^2_{\IE_1(T)} \\
	&\leq C \left( \norm{(h-\psi_0^*)}^2_{\IE_1(T)} + 2\norm{(h-\psi_0^*)}_{\IE_1(T)} \norm{\psi_0^*}_{\IE_1(T)} + \norm{\psi_0^*}^2_{\IE_1(T)}\right)\\
	&\leq r C \left(r + 2 \norm{\psi_0^*}_{\IE_1(T)} + \norm{\psi_0^*}^2_{\IE_1(T)}/r\right),
	\end{align*}	
	for $\psi=\cT_{\psi_0} h$ and  $h\in \mathds{B}_{r,T^*,\psi_0}$, and to prove contractivity, one estimates for $\psi=\cT_{\psi_0} h$ and $\psi'=\cT_{\psi_0} h'$, where $h,h'\in \mathds{B}_{r,T^*,\psi_0}$,
	\begin{align*}
	\norm{\psi-\psi'}_{\IE_1(T)} &\leq C \norm{\cF(h,h)-\cF(h',h')}_{\IE_0(T)} \\
	&\leq C \norm{h -h'}_{\IE_1(T)} \left(\norm{h-\psi_0^*}_{\IE_1(T)}+  \norm{h'-\psi_0^*}_{\IE_1(T)} + 2\norm{\psi_0^*}_{\IE_1(T)}\right)\\
	&\leq 2 C \norm{h -h'}_{\IE_1(T)} \left( r +  \norm{\psi_0^*}_{\IE_1(T)}/r \right),
	\end{align*}	
	Assuming $r>0$ and $T^*\in (0,T]$ to be sufficiently small to satisfy the formulated conditions, one has
\begin{align*}
r + 2 \norm{\psi_0^*}_{\IE_1(T)} + \norm{\psi_0^*}^2_{\IE_1(T)}/r < 1/C
\hbox{ and }  r +  \norm{\psi_0^*}_{\IE_1(T)}/r <1/2C.
\end{align*}		
Hence, the mapping $\cT_{h_0}$ restricts to a contractive self-mapping on $\mathds{B}_{r,T^{\ast},\psi_0}$ the fixed point of which is the unique local solution to  \eqref{eq:semilin}.
\end{proof}

\subsection{Strong solution for initial data in $L^2$} 
Here, one considers with $A$ as in Theorems~\ref{thm:hns} and~\ref{thm:hpe} the spaces
\begin{align*}
X_0=D(A^{-1/2}) \quad \hbox{and} \quad X_1=D(A^{1/2}). 
\end{align*}
and solves \eqref{eq:Fform1} with the non-linearities given in \eqref{eq:Fform2}. 
There are many other methods available to construct strong solutions  such as Giga, Kato or Fujita-Kato iteration schemes. Maximal $L^2$-regularity is used here for the control of the regularity class which is suitable to use solutions as test-functions for weak solutions.

\begin{lemma}[Self-adjointness for the linear part]\label{lem:max_reg_Lapalce}
	The operators $A$ in Theorems~\ref{thm:hns} and~\ref{thm:hpe} $(i)$ and $(ii)$ considered as operators in $D(A^{-1/2})$ with domain $D(A^{1/2})$
are self-adjoint. 
\end{lemma} 
\begin{proof}
	Considering the case of the Navier-Stokes equations one deduces that the operators $A$ as operators in $L^2(\Omega)^3$ are self-adjoint since their spectral resolution is given by the Fourier series expansion. Since $L^2_{\sigma}(\Omega)$ is an invariant subspace of $L^2(\Omega)^3$  with respect to $A$, also the restricted operator is self-adjoint. Hence also the operator on the interpolation-extrapolation scale is self-adjoint. For the case of the primitive equations analogous arguments hold. 
\end{proof}

\begin{lemma}[Estimates on the non-linearity for the hyper-viscous Navier-Stokes equations] \label{lem:nonlin_NS} Let $A$ be as in Theorem~\ref{thm:hns} $(i)$ or $(ii)$, then
		\begin{align*}
		\norm{\PP_{\sigma} \div (u\otimes u)}_{D(A^{-1/2})} \leq \norm{u}^2_{D(A^{1/4})}, \quad u \in D(A^{1/4}).
		\end{align*}
\end{lemma}
\begin{proof}
	In the situation of Theorem~\ref{thm:hns} $(i)$ one has $D(A^{-1/2})=H_{per,\sigma}^{-5/4}(\Omega)$ and therefore
	\begin{align*}
	\norm{\PP_{\sigma} \div (u\otimes u)}_{D(A^{-1/2})}&=	\norm{\PP_{\sigma}\div (u\otimes u)}_{H_{per}^{-5/4}} \\
	&\leq	C\norm{(u\otimes u)}_{H^{-1/4}_{per}}\\
	&\leq	C\norm{u\otimes u}_{L^{12/7}} \leq C\norm{u}^2_{L^{24/7}} \\
	&\leq
	C\norm{u}^2_{H^{5/8}} =C\norm{u}^2_{D(A^{1/4})}.
	\end{align*}
	Here, the boundedness of $\PP_{\sigma}$ in $H_{per}^s(\Omega)$,  the mapping properties of $\div \colon H^{s+1}_{per}(\Omega)^{3 }\rightarrow H^{s}_{per}(\Omega)$, $s\in \R$, the embedding $L^{12/7}(\Omega)\hookrightarrow H^{-1/4}(\Omega)$, 
	H\"older's inequality, the Sobolev embedding $H^{5/8}(\Omega)\hookrightarrow L^{24/7}(\Omega)$  and $D(A^{1/4})=H^{5/8}_{per,\sigma}(\Omega)\subset H^{5/8}(\Omega)^3$ are used.

	In the situation of Theorem~\ref{thm:hns} $(b)$ one has 
	\begin{align*}
	\norm{\PP_{\sigma} \div (u\otimes u)}_{D(A^{-1/2})}& C\leq	\norm{\PP_{\sigma} \div (u\otimes u)}_{H_{per}^{-1}} \\
	&\leq	C\norm{u\otimes u}_{L^2}\\
	&\leq	C\norm{u}^2_{L_z^{4} L_{xy}^{4}} \leq  \norm{u}^2_{H_z^{1/4} H_{xy}^{1/2}} \leq \norm{u}^2_{D(A^{1/4})}.
	\end{align*}
	Here one uses $H^{-1}_{per}(\Omega)^3\hookrightarrow D(A^{-1/2})$ which follows by duality form $D(A^{1/2})\hookrightarrow H_{per}^1(\Omega)^3$. 
	Furthermore, 
	\begin{align*}
	\norm{\PP_{\sigma} \div (u\otimes u)}_{H_{per}^{-1}}= \sup_{0\neq \varphi\in H^1(\Omega)} \frac{|\int_{\Omega}\PP_{\sigma} \div (u\otimes u)\cdot \varphi|}{\norm{\varphi}_{H_{per}^1}} \leq \norm{u\otimes u}_{L^2} \frac{\norm{\nabla \PP_{\sigma}\varphi}_{L^2}}{\norm{\varphi}_{H_{per}^1}} \leq \norm{u\otimes u}_{L^2},
	\end{align*}
	where one uses that in the periodic setting $\norm{\nabla\PP_{\sigma}\varphi}_{L^2}\leq\norm{\nabla\varphi}_{L^2}$, and an anisotropic H\"older inequality is applied as well as the embeddings $H_z^{1/4}\hookrightarrow L^4_z$ and $H_{xy}^{1/2}\hookrightarrow L^4_{xy}$. To prove in the last inequality that $D(A^{1/4}) \hookrightarrow H_z^{1/4} H_{xy}^{1/2}$ holds, one can apply the mixed derivative theorem, cf. \cite[Corollary III.4.5.10]{PruessSimonett2016}, to $-\Delta_H$ with $D(-\Delta_H)=L^2_zH_{xy}^{2}$ and $(-\partial_z^2)^{1/2}$ with $D((-\partial_z^2)^{1/2})=H^1_zL_{xy}^2$ for $\theta=1/2$, cf. \eqref{eq:mdt:xyz},	
	 to obtain  $$ D(A^{1/4})\hookrightarrow L^2_zH_{xy}^{1}\cap H_{z}^{1/2}L^2_{xy}\hookrightarrow  H_z^{1/4} H_{xy}^{1/2}.$$  Alternatively one can deduce using 
	the Fourier series representation that
	\begin{align*}
	D(A^{1/4})=\{u\in L^2_{\sigma}(\Omega) \colon \sum_{k\in \Z^3} (|k_H|^4 + |k_z|^2)^{1/2} |\hat{u}(k)|^2 <\infty    \}.
	\end{align*}
	Hence assuming without loss of generality that $\int_{\Omega}u=0$ one obtains by Parsevals's and Young's inequality,
	\begin{align*}
	\norm{u}^2_{H_z^{1/4} H_{xy}^{1/2}} &\leq C\sum_{k\in \Z^3} |k_H|\cdot |k_z|^{1/2} |\hat{u}(k)|^2 \\
	&\leq C\sum_{k\in \Z^3} (|k_H|^2 + |k_z|) |\hat{u}(k)|^2 \\
	&\leq C\sum_{k\in \Z^3} (|k_H|^4 + |k_z|^2)^{1/2} |\hat{u}(k)|^2 
	=C \norm{u}^2_{D(A^{1/4})}.\qedhere
	\end{align*} 
\end{proof}

\begin{lemma}[Estimates on the non-linearity for the hyper-viscous primitive equations] \label{lem:nonlin_PE} Let $A$ be as in Theorem~\ref{thm:hpe} $(i)$, then
		\begin{align*}
		\norm{\PP_{\bar{\sigma}} \div (u(v)\otimes v')}_{D(A^{-1/2})} \leq \norm{v}_{D(A^{13/32})}\norm{v'}_{D({A^{3/32}})}, \quad v,v' \in D(A^{1/2}).
		\end{align*}
		If $A$ is as in Theorem~\ref{thm:hpe} $(ii)$, then 
		\begin{align*}
		\norm{\PP \div (u(v)\otimes v')}_{D(A^{-1/2})} \leq \norm{v}_{D(A^{3/8})}\norm{v'}_{D({A^{1/8}})}, \quad v,v' \in D(A^{1/2}).
		\end{align*}
\end{lemma}
\begin{proof}
	In the situation of Theorem~\ref{thm:hpe} $(i)$ one has $D(A^{-1/2})=H_{\bar\sigma,per}^{-8/5}(\Omega)$. 
	Then
	\begin{align*}
	\norm{\PP_{\bar\sigma} \div (u(v)\otimes v)}_{D(A^{-1/2})}&=	\norm{\PP_{\bar\sigma} \div (u(v)\otimes v)}_{H_{per}^{-8/5}} 
	\leq	C\norm{u(v)\otimes v}_{H_{per}^{-3/5}}\\
	&\leq	C\norm{u(v)\otimes v}_{L^{10/7}} 
	\leq  C \norm{w(v)\cdot v}_{L^{10/7}} + C\norm{v\otimes v}_{L^{10/7}} \\
	&\leq C \norm{w(v)}_{L_z^{\infty}L_{xy}^{20/7}} \norm{v}_{L_z^{10/7}L_{xy}^{20/7}} + C \norm{v}_{L_z^{2}L_{xy}^{5}}\norm{v}_{L_z^{5}L_{xy}^{2}}\\
	&\leq C \norm{v}_{L_z^{2}H_{xy}^{13/10}} \norm{v}_{L_z^{2}H_{xy}^{3/10}} + C \norm{v}_{L_z^{2}H_{xy}^{3/5}}\norm{v}_{H_z^{3/10}L_{xy}^{2}}\\
	&\leq C \norm{v}_{D(A^{13/32})} \norm{v}_{D(A^{3/32})}.
	\end{align*}
	Here, one has used boundedness of $\PP_{\bar\sigma}$ in $H^s(\Omega)$, that $\div \colon H^{s+1}(\Omega)^3\rightarrow H^{s}(\Omega)$, $s\in \R$, and the embedding $L^{10/7}(\Omega)\hookrightarrow  H^{-3/5}(\Omega)$. Furthermore, anisotropic H\"older estimates are applied, where it is used that $$\norm{w(v)}_{L_z^{\infty}L_{xy}^{20/7}}\leq C \norm{\partial_z w(v)}_{L_z^{2}L_{xy}^{20/7}}=C \norm{\div_H v}_{L_z^{2}L_{xy}^{20/7}},$$ 
	where the embedding $H^1_z\hookrightarrow L^{\infty}_z$ and Poincar\'{e}'s inequality for $w(v)$ have been applied. Moreover, the embeddings $H_{xy}^{3/10}\hookrightarrow L_{xy}^{20/7}$, $H_{xy}^{3/5}\hookrightarrow L_{xy}^{5}$ and $H_z^{3/10}\hookrightarrow L_z^{5}$ are applied. Eventually,  $D(A^{13/32}) \hookrightarrow H^{13/10}(\Omega)\hookrightarrow  L_z^{2}H_{xy}^{13/10}$ and $D(A^{3/32})\hookrightarrow  H^{3/10}(\Omega)\hookrightarrow H_z^{3/10}L_{xy}^{2}$.

	In the situation of Theorem~\ref{thm:hpe} $(ii)$ one has 
	\begin{align*}
	\norm{\PP_{\bar\sigma} \div (u(v)\otimes v)}_{D(A^{-1/2})} &\leq 
	\norm{\PP_{\bar\sigma} \div (u(v)\otimes v)}_{H_{per}^{-1}} \\
	& \leq	\norm{u(v)\otimes v}_{L^2} \\
	&\leq	C\norm{v}_{L^2_z H_{xy}^{3/2}} \norm{v}_{L^2_z H_{xy}^{1/2}} + 	C\norm{v}_{H^{1/2}_z H_{xy}^{1/2}} \norm{v}_{L^2_z H_{xy}^{1/2}} \\
	&\leq C \norm{v}_{D(A^{3/8})} \norm{v}_{D(A^{1/8})}, 	
	\end{align*}
	where similar arguments as in Lemma~\ref{lem:nonlin_NS} have been applied together with the estimates 
	\begin{align*}
	\norm{w(v)v}_{L^2}
	\leq C \norm{w(v)}_{L^{\infty}_zL^4_{xy}}\norm{v}_{L^{2}_zL^4_{xy}}
	\leq C \norm{\nabla_H v}_{L^{2}_zH^{1/2}_{xy}}\norm{v}_{L^{2}_zH^{1/2}_{xy}}
	\leq C\norm{v}_{H^{1/8}_z H_{xy}^{1/4}} \norm{v}_{H^{3/8}_z H_{xy}^{3/4}}
	\end{align*}
	and
	\begin{align*}
	\norm{v\otimes v}_{L^2}
	\leq C \norm{v}_{L^{3/8}_zL^{3/8}_{xy}}\norm{v}_{L^{8}_zL^8_{xy}}
	\leq C \norm{v}_{H^{1/8}_zH^{1/4}_{xy}}\norm{v}_{H^{3/8}_zH^{3/4}_{xy}}
	\leq C \norm{v}_{D(A^{3/8})} \norm{v}_{D(A^{1/8})},
	\end{align*}
	where Sobolev embeddings and the mixed derivative theorem are applied to obtain
\begin{align*}
D(A^{1/8})\hookrightarrow H_z^{1/4}L^2_{xy}\cap L_z^{2}H^{1/2}_{xy} \hookrightarrow H^{1/8}_zH^{1/4}_{xy} \quad \hbox{and} \quad D(A^{3/8})\hookrightarrow H_z^{3/4}L^2_{xy}\cap L_z^{2}H^{3/2}_{xy} \hookrightarrow H^{3/8}_zH^{3/4}_{xy},
\end{align*}
which, as in the proof of Lemma~\ref{lem:nonlin_NS}, can be proven alternatively by using the Fourier series representation. 
\end{proof}

In the following denote by $a(k)$, $k\in \Z^3$,  the (elliptic) symbol of $A$, where $a(k)>0$ for $k\neq 0$, that is
\begin{align*}
a(k)= \nu|k|^2 +\varepsilon |k|^{5/4}, \quad a(k)= \nu|k|^2 +\varepsilon |k|^{8/5} \quad \hbox{or}\quad a(k)= \nu|k|^2 +\varepsilon |k_H|^{4}, \quad k_H=(k_1,k_2).
\end{align*}

\begin{proposition}[Global strong well-posedness for the hyper-viscous Navier-Stokes equations] \label{prop:loc_NS}
	Let the operator $A$ be as in Theorem~\ref{thm:hns} $(i)$ or $(ii)$. Then for each
		\begin{align*}
		u_0\in L^2_{\sigma}(\Omega) \quad \hbox{and} \quad f\in L^2(0,T;D(A^{-1/2})), 
		\end{align*}
		there is a unique solution $u\in \IE_1(T)$ to \ref{eq:NS} and it satisfies the energy equality
\begin{align}\label{eq:NS_energy_eq}
\norm{u(t)}_{L^2(\Omega)}^2 + 2 \int_0^t \norm{A^{1/2}u(s)}_{L^2(\Omega)}^2 ds =    \norm{u(0)}_{L^2(\Omega)}^2 + 2 \int_0^t f(s)\cdot u(s) ds, \quad 0<t\leq T.
\end{align}
\end{proposition}	
	\begin{proof}
		Using Lemma~\ref{lem:max_reg_Lapalce}, Remark~\ref{rem:Hilbert} and Lemma~\ref{lem:nonlin_NS} local existence on $(0,T^*)$ for some $T^*\in (0,T]$ follows by Proposition~\ref{prop:loc_ex}. Moreover, the solution is regular enough to test it with itself which gives the energy equality on $(0,T^*)$ for any existence time $T^*\in (0,T]$.
%

To prove now global existence, explicit control on the existence time is used.
Note that by maximal regularity there exists a constant $C>0$ such that 
\begin{align*}
\norm{u_0^*}^2_{\IE_1(T^*)}\leq  
C\norm{f}^2_{\IE_0(T^*)} +\norm{e^{-tA}u_0}^2_{\IE_1(T^*)} \quad \hbox{and} \quad  \norm{e^{-tA}u_0}^2_{\IE_1(T^*)}= \norm{(A+1)e^{-tA}u_0}^2_{\IE_0(T^*)}.
\end{align*}
By Proposition~\ref{prop:loc_ex} for $r=1/4C$ one has for $\norm{u_0^*}^2_{\IE_1(T^*)}< B(r,C)$
that $u\in \mathds{B}_{r,T^*,u_0}$. So first, note that since $\norm{f}^2_{\IE_0(T)}<\infty$, for any $\delta>0$, there is a $\tau>0$ and $m\in \N$ such that $T=\tau m$, and with $I_j=((j-1)\tau,j\tau)$, $1\leq j\leq m$, such that $\norm{f}^2_{\IE_0(I_j)}<\delta$. Hence for $\delta=(1/2C)B(r,C)$, one has a partition into finitely many intervals $I_j$ where 
$\norm{u_0^*}^2_{\IE_1(I_j)}\leq \frac{1}{2} B(r,C) + \norm{e^{-tA}u_0}^2_{\IE_1(I_j)}$. 

Now, the aim is to find for each $I_j$ a partition into finitely many subintervals $I$, where $\norm{e^{-tA}u_0}^2_{\IE_1(I)}<\frac{1}{2} B(r,C)$. To this end, compute for an interval $I=[\tau,\sigma]\subset I_j$ using Parsevals's equality
	\begin{align*}
	\norm{(A+1)e^{-tA}u(\tau)}^2_{\IE_0(I)} &= \int_{\tau}^{\sigma}\sum_{k\in \Z^3} \frac{(a(k)+1)^2}{(a(k)+1)} e^{-2ta(k)}|\hat{u}(\tau,k)|^2 dt \\
	&= \sum_{k\in \Z^3\setminus\{0\}} \frac{(a(k)+1)}{-2a(k)} (e^{-2a(k)\sigma}-e^{-2a(k)\tau}) |\hat{u}(\tau,k)|^2 dt  + \int_{I}|\hat{u}(\tau,0)|^2 dt  \\
	&= \sum_{k\in \Z^3\setminus\{0\}} \frac{(a(k)+1)}{2a(k)} (1-e^{-2a(k)|I|})e^{-2a(k)\tau} |\hat{u}(\tau,k)|^2 dt  + \int_{I}|\hat{u}(\tau,0)|^2 dt  \\
	&\leq C \max\{1, 2C_a |k|_{\min}\}|I|\cdot \left(\norm{u(0)}^2_{L^2(\Omega)} + \norm{f}^2_{L^2(0,T;D(A^{-1/2}))}\right)
	\end{align*}
	where $a(k)$ denotes the symbol of $A$,
\begin{align*}
|k|_{\min}:=\min\{|a(k)|\colon k\in \Z^3\setminus\{0\} \} \quad\hbox{and}\quad C_a=\max_{k\in \Z^3\setminus\{0\}}\left\{\frac{(a(k)+1)}{2a(k)}e^{-2a(k)T}\right\},
\end{align*}	
and one uses that $|1-e^{-2a(k)t}|\leq 2tk_{\min}$ for $2tk_{\min}\leq 1$. In the last inequality \eqref{eq:energyestimate} has been used to estimate $\norm{u(\tau)}^2_{L^2}$ and  $\int_{I}|\hat{u}(\tau,0)|^2 dt\leq |I| \norm{u}^2_{L^{\infty}(0,T;L^2)}$. 
Hence, with $C_1=C\max\{1, 2C_a |k|_{\min}\}$ and
\begin{align*}
|I| = \min\left\{
\frac{1}{4C_1 \left(\norm{u(0)}^2_{L^2(\Omega)} + \norm{f}^2_{L^2(0,T;D(A^{-1/2}))}\right)} B(r,C), 1/2k_{\min}
\right\}\hbox{ one has } \norm{e^{-tA}u_0}^2_{\IE_1(I)} \leq \frac{1}{4}B(r,C),
\end{align*}
and since this is a uniform time step global existence follows by induction in finitely many steps. 
\end{proof}

The proof for the case of the hyper-viscous primitive equations is analogous to the above using Lemma~\ref{lem:nonlin_PE} instead of Lemma~\ref{lem:nonlin_NS}. 
\begin{proposition}[Global strong well-posedness for the hyper-viscous primitive equations]\label{prop:loc_PE} 
	Let the operator $A$ be as in Theorem~\ref{thm:hpe} $(i)$ or $(ii)$. Then for each
	\begin{align*}
	v_0\in L^2_{\overline{\sigma}}(\Omega) \quad \hbox{and} \quad f\in L^2(0,T;D(A^{-1/2})), 
	\end{align*}
	there is a unique solution $v\in \IE_1(T)$ to \ref{eq:PE}, and 
it satisfies the energy equality
	\begin{align}\label{eq:PE_energy_eq}
	\norm{v(t)}_{L^2(\Omega)}^2 + 2 \int_0^t \norm{A^{1/2}v(s)}_{L^2(\Omega)}^2 ds =    \norm{v(0)}_{L^2(\Omega)}^2 + 2 \int_0^t f(s)\cdot v(s) ds, \quad 0<t\leq T.
	\end{align}
\end{proposition}	

\subsection{Weak-strong uniqueness}
\begin{proof}[Proof of Theorem~\ref{thm:hns} and \ref{thm:hpe}(a)]
Let $u_s$ be the global strong solution given in Proposition~\ref{prop:loc_NS}, and $u_w$ be any weak solution to \eqref{eq:NS} both for data $u_0$ and $f$.
Using the regularity of strong solution $u_s$ and an approximation argument by smooth functions with respect to time, the strong solution can be inserted as test function in the definition of weak solutions $u_w$, i.e, one obtains for $t\in (0,T)$ multiplying it by two
\begin{align*}
	-2\int_0^t u_w \cdot \partial_t u_s + A^{1/2}u_w\cdot A^{1/2}u_s - (u_w\cdot \nabla u_w)\cdot u_s = 2u_w(t)\cdot u_s(t)- 2\norm{u(0)}_{L^2}^2 - 2\int_0^t f\cdot u_s.
\end{align*}
Adding amongst others this to the energy inequality~\ref{eq:ei_ns} and the energy equality from Proposition~\ref{prop:loc_NS} and using further  standard arguments one obtains for 
 $\delta=u_s-u_w$ with $\delta(0)=0$, cf. \cite[Proof of Theorem X.3.1]{Galdi}
\begin{align*}
\norm{\delta(t)}^2_{L^2} + 2\int_0^t \norm{A^{1/2}\delta(s)}^2_{L^2} ds \leq 	2\int_0^t (\delta\cdot \nabla \delta)\cdot u_s.
\end{align*}
It remains now to estimate the right-hand side. If $A$ is as in Theorem~\ref{thm:hns} (i) then one estimates
\begin{align*}
\int_0^t (\delta\cdot \nabla \delta)\cdot u_s 
&\leq \int_0^t \norm{\delta}_{L^2}
\norm{\delta}_{L^{12}}  \norm{\nabla u_s}_{L^{12/5}} \\
&\leq C\int_0^t\norm{\delta}_{L^{2}}\norm{\delta}_{H^{5/4}} \norm{\nabla u_s}_{H^{1/4}} \\
&\leq C \int_0^t\norm{\delta}^2_{L^{2}}\norm{u_s}_{D(A^{1/2})}^2 + \int_0^t\norm{\delta}_{D(A^{1/2})}^2 \\
&\leq C \int_0^t\norm{\delta}^2_{L^{2}}(\norm{u_s}_{D(A^{1/2})}^2+1/C) + \int_0^t\norm{A^{1/2}\delta}_{L^2}^2.
\end{align*}
Now, one can compensate the term $\int_0^t\norm{A^{1/2}\delta}_{L^2}^2$ into the left-hand side to obtain 
\begin{align*}
\norm{\delta(t)}^2_{L^2} + \int_0^t \norm{A^{1/2}\delta(s)}^2_{L^2} ds \leq 	C \int_0^t\norm{\delta}^2_{L^{2}}(\norm{u_s}_{D(A^{1/2})}^2+1/C), 
\end{align*}
and then the claim follows using the regularity of the strong solutions and Gr\"onwall's inequality.

If $A$ is as in Theorem~\ref{thm:hns} (ii) one estimates as in the proof of Lemma~\ref{lem:nonlin_PE} 
\begin{align*}
\int_0^t (\delta\cdot \nabla \delta)\cdot u_s 
&\leq C\int_0^t \norm{\delta}_{D(A^{1/4})}^2
 \norm{\nabla u_s}_{L^{2}} 
\leq \int_0^t \norm{\delta}_{D(L^2)} \norm{\delta}_{D(A^{1/2})}
 \norm{\nabla u_s}_{L^{2}} \\
&\leq C\int_0^t\norm{\delta}^2_{L^{2}}\norm{u_s}_{D(A^{1/2})}^2 + \int_0^t\norm{\delta}_{D(A^{1/2})}^2 
\end{align*}
using the interpolation inequality $\norm{\delta}_{D(A^{1/4})}^2\leq  \norm{\delta}_{L^2}\norm{\delta}_{D(A^{1/2})}$. The claim follows then as above.

Similarly for the primitive equations, let
$\delta=v_s-v_w$, where $v_w$ denotes any weak solution and $v_s$ the strong solution from Proposition~\ref{prop:loc_PE}, then one shows, analogously to \cite[Section 5]{GaldiHieberKashiwabara2017}, that
\begin{align}\label{eq:wekstrong_ns}
\norm{\delta(t)}^2_{L^2} + 2\int_0^t \norm{A^{1/2}\delta(s)}^2_{L^2} ds \leq 	\int_0^t (\delta\cdot \nabla \delta)\cdot v_s 
\end{align}
and estimates in the case of $A$ as in Theorem~\ref{thm:hpe} (i)
\begin{align*}
\int_0^t (\delta\cdot \nabla \delta)\cdot u_s 
&\leq \int_0^t \norm{\delta}_{L^2}
\norm{\delta}_{L^{5}}  \norm{\nabla u_s}_{L^{10/3}} 
\leq C\int_0^t\norm{\delta}_{L^{2}}\norm{\delta}_{H^{9/10}} \norm{\nabla u_s}_{H^{3/5}} \\
&\leq C\int_0^t\norm{\delta}^2_{L^{2}}\norm{u_s}_{D(A^{1/2})}^2 + \int_0^t\norm{\delta}_{D(A^{1/2})}^2. 
\end{align*}
If $A$ is as in Theorem~\ref{thm:hpe} (ii) as in the proof of Lemma~\ref{lem:nonlin_PE} 
\begin{align*}
\int_0^t (\delta\cdot \nabla \delta)\cdot u_s 
&\leq C\int_0^t \norm{\delta}_{D(A^{1/8})}\norm{\delta}_{D(A^{3/8})}
 \norm{\nabla u_s}_{L^{2}} 
\leq C\int_0^t \norm{\delta}_{D(L^2)} \norm{\delta}_{D(A^{1/2})}
 \norm{\nabla u_s}_{L^{2}} \\
&\leq C\int_0^t\norm{\delta}^2_{L^{2}}\norm{u_s}_{D(A^{1/2})}^2 + \int_0^t\norm{\delta}_{D(A^{1/2})}^2 
\end{align*}
using the interpolation inequalities $\norm{\delta}_{D(A^{1/8})}\leq  \norm{\delta}_{L^2}^{3/4}\norm{\delta}_{D(A^{1/2})}^{1/4}$ and $\norm{\delta}_{D(A^{3/8})}\leq  \norm{\delta}_{L^2}^{1/4}\norm{\delta}_{D(A^{1/2})}^{3/4}$.
The claim now follows as above by a compensation argument, the regularity of strong solutions and Gr\"onwall's inequality.
\end{proof}

\subsection{Convergence}
Note first that by the energy equality the sequences $(u_{\varepsilon_n})$ and $(v_{\varepsilon_n})$
of solutions to \eqref{eq:NS} and  \eqref{eq:PE}, respectively, are uniformly bounded in $C^0([0,T];L^2(\Omega))\cap L^2(0,T; H^1(\Omega))$ by \eqref{eq:energyestimate}. Hence, there are weakly convergent subsequences in $C^w(0,T;L^2(\Omega))\cap L^2(0,T; H^1(\Omega))$. Strong convergence is now proven via a compactness argument.
\begin{theorem}[Aubin-Lions Lemma, cf. Chapter III, Theorem 2.1 in \cite{Temam1977}]\label{thm:Temam}
	Let $X_0\subset X\subset X_1$ be Banach spaces such that
	\begin{itemize}
	\item[(i)] $X_0,X_1$ are reflexive,
	\item[(ii)] the injection $X_0\hookrightarrow X$ is compact.
	\end{itemize}
	Let $T\in (0,\infty)$ and $\alpha_0,\alpha_1>1$, then the embedding
	 \begin{align*}
	 \{v\in L^{\alpha_0}(0,T;X_0)\colon v_t\in L^{\alpha_1}(0,T;X_1) \}  \hookrightarrow L^{\alpha_0}(0,T;X)
	 \end{align*}
	 is compact.
\end{theorem}

Set $A_H= 1/\varepsilon (A + \nu \Delta)$, $\nu,\varepsilon>0$, i.e. $A= -\nu \Delta + \varepsilon A_H$. Consider first the solutions to the hyper-viscous Navier-Stokes equations.
\begin{lemma}\label{lem:ns_timereg}
Let $u_{\varepsilon_n}$, $\varepsilon_n>0$, as in Theorem~\ref{thm:hns} $(b)$,  then for $s>5/2$
\begin{align*}
\norm{\partial_t u_{\varepsilon_n}}_{L^{2}(0,T;H^{-s})} \leq C(\norm{u_0}_{L^2}, \norm{f}_{L^2(0,T;H^{-1})}),  \quad n\in \N.
\end{align*}
\end{lemma}
\begin{proof}
Let 
	\begin{align*}
X_0=H^{-s}_{per,\sigma}(\Omega) \quad \hbox{and}\quad X_1= \{u\in X_0 \colon \frac{|a(k)|^2}{(1+|k|^2)^s} |\hat{u}(k)|^2 <\infty\},
	\end{align*}
then $A$ with domain $D(A)=X_1$ defines a self-adjoint operator in $X_0$.  	
One defines for a Hilbert space $X$ an extension operator  $$E_{\chi}= \chi \cdot E_{per} \circ E_{even}\colon H^1(0,T;X)\rightarrow H^1(\R;X)$$ by even reflection $E_{even}\colon H^1(0,T;X)\rightarrow H^1(-T,T;X)$ at $0$, periodic continuation by $E_{per}\colon H^1(-T,T;X)\rightarrow H^1(\R;X)$ and multiplication by a sufficiently smooth cut-off function $\chi$ with $\supp \chi \subset [-2T,2T]$.
This operator  extends to the corresponding $L^2$-space.  Moreover,
\begin{align}\label{eq:echiu}
(\partial_t  + A)E_{\chi}u_{\varepsilon_n} = \chi'E_{per}E_{even}u_{\varepsilon_n} + E_{\chi}\cF(u_{\varepsilon_n},u_{\varepsilon_n}) + E_{\chi}f.
\end{align}
Now, taking Fourier transform in time and Fourier series in space, one obtains 
 \begin{align*}
\norm{(\partial_t  + A)E_{\chi}u_{\varepsilon_n}}_{\IE_0(\R)}= C \sum_{k\in \Z^3}\int_{\R}|i\tau + a(k)|^2|\widehat{E_{\chi}u}_{\varepsilon_n}(\tau,k)|^2 d\tau
\geq \norm{\partial_t E_{\chi}u}_{\IE_0(\R)} \geq c\norm{\partial_tu_{\varepsilon_n}}_{\IE_0(T)},
\end{align*}
for some $C,c>0$, 
using the ellipticity of the symbol $a(k)$ of $A$, where
\begin{align*}
\hat{\varphi}(\tau,k) = \frac{1}{2\pi}\int_{\R}\int_{\Omega} \varphi(t,x,y,z) e^{ik\cdot(x,y,z)}e^{i\tau t} dx\,dy\, dz\,dt.
\end{align*}
Hence by \eqref{eq:echiu} and using the boundedness of the extension operator
 \begin{align*}
 \norm{\partial_tu_{\varepsilon_n}}_{\IE_0(T)} \leq C \left(\norm{u_{\varepsilon_n}}_{\IE_0(T)} + \norm{\cF(u_{\varepsilon_n},u_{\varepsilon_n})}_{\IE_0(T)}+  \norm{f}_{\IE_0(T)} \right).
\end{align*} 
Choosing $s>5/2$, one has $L^1(\Omega)\hookrightarrow H^{-s+1}_{per,\sigma}(\Omega)$, and hence using \eqref{eq:energyestimate}
\begin{align*}
\norm{\cF(u_{\varepsilon_n},u_{\varepsilon_n})}_{L^2(0,T;H^{-s})} 
&= \left(\int_0^T\norm{\PP_{\sigma}\div (u_{\varepsilon_n}\otimes u_{\varepsilon_n})}^2_{H^{-s}}\right)^{1/2}
\leq \norm{u_{\varepsilon_n}}^2_{L^2(0,T;L^2)} \\
&\leq C T
(\norm{u(0)}^2_{L^2} + \norm{f}^2_{L^{2}(0,T;H^{-1})}), \quad n\in \N.
\end{align*}
	Also, using the embeddings $H^{-s}_{per,\sigma}(\Omega)\hookrightarrow L^2(\Omega)$ and $D(A^{-1/2})\hookrightarrow H^{-1}_{per,\sigma}(\Omega)$ one obtains for some $C>0$
\begin{align*}
\norm{u_{\varepsilon_n}}_{\IE_0(T)} \leq C T \norm{u_{\varepsilon_n}}^2_{L^{\infty}(0,T;L^2)} \quad \hbox{and} \quad \norm{f}_{\IE_0(T)} \leq C \norm{f}_{L^{2}(0,T;H^{-1})}
\end{align*}	
which concludes the proof using again 	\eqref{eq:energyestimate}.
\end{proof}
\begin{proof}[Proof of Theorem~\ref{thm:hns} $(b)$]
First, by Lemma~\ref{lem:ns_timereg} one can apply Theorem~\ref{thm:Temam} with $\alpha_0=2$, $\alpha_1=2$ and
\begin{align*}
X_1=H^{-s}_{per, \sigma}(\Omega), \quad X_0= H^{1}_{per, \sigma}(\Omega), \quad X = H^{1-\delta}_{per, \sigma}(\Omega), \quad \delta \in (0,1], \quad s>5/2.
\end{align*}
Hence strong convergence follows from the compact embedding.

Second, one shows that the energy inequality holds. The strong convergence above implies convergence in $L^2(0,T;H^{-1})$  and therefore for fixed test function $\varphi$ one has
\begin{align*}
\int_s^t\varepsilon (A_H)^{1/2}u_{\varepsilon} \cdot (A_H)^{1/2}\varphi \leq \varepsilon\norm{u_\varepsilon}_{L^2(0,T; L^2(\Omega)))} \norm{\varphi}_{L^2(0,T; D(A))} \to 0 \hbox{ as } \varepsilon\to 0,
\end{align*} 
where $A_H=\Delta_H^2$ or $A_H=(-\Delta)^{5/4}$, and furthermore due to the weak convergence
\begin{align*}
\int_0^t f(\tau)\cdot u_{\varepsilon}(\tau) \hbox{d} \tau \to \int_0^t f(\tau)\cdot u(\tau) \hbox{d} \tau \hbox{ as } \varepsilon\to 0. 
\end{align*}
Hence the energy inequality follows from weak convergence, and  the limit function satisfies the Navier-Stokes equations in the weak sense for test functions $\varphi \in C^{1}([0,T]; D(A^{1/2}))\cap L^2(0,T; D(A))$. In particular for the non-linear term it follows from the strong convergence in $L^2(0,T; H^{3/4}(\Omega))$ that 
\begin{multline*}
\int_0^t[(u_\varepsilon\cdot \nabla u_\varepsilon)-(u\cdot \nabla u)]\cdot \varphi \\
\leq
C  \norm{u_\varepsilon- u}_{L^2(0,T; H^{3/4})}(\norm{u_\varepsilon}_{L^2(0,T; H^{3/4})}+ \norm{u}_{L^2(0,T; H^{3/4})} \norm{\nabla \varphi}_{L^2(0,T; L^2)}) \to 0 \hbox{ as } \varepsilon\to 0.
\end{multline*} The various terms define continuous functionals even for  $\varphi \in C^{1}([0,T]; D((-\PP_{\sigma}\Delta)^{1/2}))\cap L^2(0,T; D(-\PP_{\sigma}\Delta))$.
By continuity and density of $D(A)\subset D(-\PP_{\sigma} \Delta)$ this carries over to the larger class of test functions $\varphi \in C^{1}([0,T]; D((-\PP_{\sigma}\Delta)^{1/2}))\cap L^2(0,T; D(-\PP_{\sigma}\Delta))$.  
\end{proof}


\begin{lemma}\label{lem:pe_timereg}
Let $v_{\varepsilon_n}$, $\varepsilon>0$ as in Theorem~\ref{thm:hpe} $(b)$, then for $s>5/2$
\begin{align*}
\norm{\partial_t v_{\varepsilon_n}}_{L^{2}(0,T;H^{-s})} \leq C(\norm{u_0}_{L^2}, \norm{f}_{L^2(0,T;H^{-1})}), \quad n\in \N.
\end{align*}
\end{lemma}
\begin{proof}
The proof is similar to the one of Lemma~\ref{lem:pe_timereg}, and here one
estimates the non-linearity by
\begin{align*}
\int_0^T\norm{\PP_{\bar\sigma}\div (u(v_{\varepsilon_n})\otimes v_{\varepsilon_n})}_{H^{-s}}^2 &\leq \int_0^T\norm{u(v_{\varepsilon_n})\otimes v_{\varepsilon_n}}_{L^1} \leq \int_0^T\norm{w(v_{\varepsilon_n})}^2_{L^2}\norm{v_{\varepsilon_n}}^2_{L^2} + \norm{v_{\varepsilon_n}}_{L^2}^4  \\
&\leq C \norm{v_{\varepsilon_n}}^2_{L^{\infty}(0,T; L^2)} \norm{v_{\varepsilon_n}}^2_{L^2(0,T; H^1)}
+  C \norm{v_{\varepsilon_n}}^4_{L^{\infty}(0,T; L^2)} \\
&\leq C(\norm{v_0}^2_{L^2}+ \norm{f}^2_{L^2(0,T;H^{-1})})^2, \quad n\in\N.
\end{align*}
Here, one has used that $\PP_{\bar\sigma}$ is bounded in $H^s_{per}(\Omega)$.  Hence, $\norm{\partial_tv_{\varepsilon}}^2_{L^2(0,T;H^{-s}_{per})}$ is uniformly bounded.
\end{proof}

\begin{proof}[Proof of Theorem~\ref{thm:hpe} $(b)$]
By Lemma~\ref{lem:ns_timereg} one can apply Theorem~\ref{thm:Temam} with $\alpha_0=2$, $\alpha_1=2$ and
\begin{align*}
X_1=H^{-s}_{per, \bar\sigma}(\Omega), \quad X_0= H^{1}_{per, \bar\sigma}(\Omega), \quad X = H^{1-\delta}_{per, \bar\sigma}(\Omega), \quad \delta \in (0,1], \quad s>5/2
\end{align*}
Hence strong convergence follows from the compact embedding.

The energy inequality holds since the strong convergence implies convergence in $L^2(0,T;H^{-1}(\Omega))$  and therefore for fixed test function $\varphi$ one has
\begin{align*}
\int_s^t\varepsilon (A_H)^{1/2}v_{\varepsilon} \cdot (A_H)^{1/2}\varphi \leq \varepsilon\norm{v_\varepsilon}_{L^2(0,T; L^2(\Omega)))} \norm{\varphi}_{L^2(0,T; D(A))} \to 0 \hbox{ as } \varepsilon\to 0,
\end{align*} 
where $A_H=\Delta_H^2$ or $A_H=(-\Delta)^{8/5}$. 
Hence the energy inequality follows from weak convergence, and  the limit function satisfies the primitive equations in the weak sense for $\varphi \in C^{1}([0,T]; D(A^{1/2}))\cap L^2(0,T; D(A))$. This extends by continuity to   $\varphi \in C^{1}([0,T]; D((-\PP_{\bar\sigma}\Delta)^{1/2}))\cap L^2(0,T; D(-\PP_{\bar\sigma}\Delta))$. 
In particular for the non-linear term it follows from the regularity of $v$ that
\begin{align*}
\int_0^t(u(v)\cdot \nabla v)\cdot \varphi 
&\leq
\norm{w(v) v + v\otimes v}_{L^2(0,T);L^1)}
\norm{\varphi}_{L^{2}(0,T;H^2)}
\\
&\leq C \norm{v}_{L^2(0,T;H^1)}\norm{v}_{L^\infty(0,T;L^2)}
\norm{\varphi}_{L^{2}(0,T;H^2)}, 
\end{align*}
where one uses anisotropic H\"older's inequalities with respect to time and space and $H^2(\Omega)\hookrightarrow L^{\infty}(\Omega)$.
By density of $D(A)\subset D(-\PP_{\bar\sigma} \Delta)$ this carries over to the larger class of test functions. 
\end{proof}

\subsection*{Acknowledgement} 
I would like to thank Daniel Reinert from the German Weather Service (DWD) for helpful insights into the functionality of meteorological simulations and in particular for pointing out the references \cite{COSMO} and \cite{Lauritzenetall}.

\end{document}